\documentclass[12pt]{article}
\usepackage{graphicx}
\usepackage{amsmath}
\usepackage{amssymb}
\usepackage{theorem}

\numberwithin{equation}{section}

\newtheorem{thm}{Theorem}[section]
\newtheorem{lemma}[thm]{Lemma}
\newtheorem{prop}[thm]{Proposition}
\newtheorem{cor}[thm]{Corollary}
{\theorembodyfont{\rmfamily}
\newtheorem{defn}[thm]{Definition}

\newtheorem{rmk}[thm]{Remark}
}

\newcommand{\qed}{\hfill \mbox{\raggedright \rule{.07in}{.1in}}}
 
\newenvironment{proof}{\vspace{1ex}\noindent{\bf
Proof}\hspace{0.5em}}{\hfill\qed\vspace{1ex}}
\newenvironment{pfof}[1]{\vspace{1ex}\noindent{\bf Proof of
#1}\hspace{0.5em}}{\hfill\qed\vspace{1ex}}

\newcommand{\R}{{\mathbb R}}

\newcommand{\Z}{{\mathbb Z}}

\newcommand{\N}{{\mathbb N}}
\newcommand{\E}{{\mathbb E}}
\renewcommand{\P}{{\mathbb P}}

\newcommand{\cB}{{\mathcal B}}

\newcommand{\cG}{{\mathcal G}}
\newcommand{\cM}{{\mathcal M}}

\newcommand{\supI}{{\SMALL \sup_{[0,1]}}}

\newcommand{\tV}{{\widetilde V}}

\newcommand{\bW}{{\overline W}}

\newcommand{\eps}{\epsilon}

\newcommand{\Lip}{{\operatorname{Lip}}}

\newcommand{\dist}{\operatorname{dist}}

\newcommand{\SMALL}{\textstyle}
\newcommand{\BIG}{\displaystyle}

\title{Rate of Convergence in the Weak Invariance Principle
for Deterministic Systems}

\author{
Marios Antoniou
\thanks{Mathematics Institute, University of Warwick, Coventry, CV4 7AL, UK}
\and 
Ian Melbourne\thanks{Mathematics Institute, University of Warwick, Coventry, CV4 7AL, UK}
}

\date{4 June 2018, revised 15 October 2018}

\begin{document}

 \maketitle

\begin{abstract}
We obtain the first results on convergence rates in the Prokhorov metric for the weak invariance principle (functional central limit theorem) for deterministic dynamical systems.  Our results hold for uniformly expanding/hyperbolic (Axiom A) systems, as well as nonuniformly expanding/hyperbolic  systems such as dispersing billiards, H\'enon-like attractors, Viana maps and intermittent maps.  As an application, we obtain convergence rates for deterministic homogenization in multiscale systems.
\end{abstract}

 \section{Introduction} 
 \label{sec:intro}

There is considerable interest in proving statistical properties for large classes of dynamical systems.  The central limit theorem (CLT) was proved for uniformly hyperbolic (Axiom~A) diffeomorphisms and flows in~\cite{Ratner73} and for various nonuniformly expanding/hyperbolic maps in~\cite{HofbauerKeller82}.\footnote{For the remainder of the introduction, we write (non)uniformly hyperbolic rather than (non)uniformly expanding/hyperbolic.}
The latter reference also established the weak invariance principle (WIP), otherwise known as the functional CLT, generalizing the classical result of Donsker~\cite{Donsker51} for independent and identically distributed random variables.
See also~\cite{DenkerPhilipp84} for the WIP for uniformly hyperbolic flows.
More recently, the WIP was obtained for large classes of nonuniformly hyperbolic maps and flows~\cite{BMsub,Gouezel04,MN05,MV16,MZ15}.
The WIP was applied in~\cite{GM13b,MS11} to obtain results on homogenization for deterministic fast-slow systems.

An important question, especially bearing in mind applications to fast-slow systems, is to obtain convergence rates for these statistical limit laws.
In the case of the CLT, these convergence rates are called
Berry-Esseen estimates;   sharp results have been obtained for uniformly hyperbolic diffeomorphisms~\cite{CoehloParry90} and nonuniformly hyperbolic systems~\cite{Gouezel05}.
However, there are few results on convergence rates in the WIP for dependent random variables in the probability theory literature and apparently none in the dynamical systems literature.  
In this paper, we obtain convergence rates in the WIP for uniformly and nonuniformly hyperbolic dynamical systems.  

For uniformly hyperbolic dynamical systems (including Axiom~A diffeomorphisms but also maps with infinitely many branches such as the Gauss map), we obtain the convergence rate $n^{-(\frac14-\delta)}$ in the WIP for $\delta$ arbitrarily small.  This result also applies to certain nonuniformly hyperbolic systems (those modelled by a Young tower with exponential tails~\cite{Young98}) such as dispersing planar periodic billiards, unimodal maps, and H\'enon-like maps~\cite{BenedicksYoung00}.

More generally, for nonuniformly hyperbolic systems the rate depends on the degree of nonuniformity.  As an indicative example, we consider
Pomeau-Manneville intermittent maps~\cite{PomeauManneville80}, specifically the map
\begin{align} \label{eq:LSV}
T:[0,1]\to [0,1], \qquad T(x)=\begin{cases} x(1+2^\gamma x^\gamma) & x\in[0,\frac12) \\ 2x-1 & x\in[\frac12,1] \end{cases},
\end{align} 
 studied in~\cite{LSV99}.
Here $\gamma>0$ is a real parameter and there is a unique ergodic absolutely continuous invariant probability measure $\mu$ for $\gamma<1$.  Moreover, H\"older observables satisfy the CLT and WIP provided $\gamma<\frac12$.
By~\cite{Gouezel05}, the convergence rate in the CLT is $\begin{cases}  n^{-\frac12} &
\gamma\in(0,\frac13) \\ 
 n^{-(\frac{1}{2\gamma}-1)} &
\gamma\in(\frac13,\frac12) \end{cases}$.
For the WIP, 
we obtain the rate $n^{-(\frac14(1-2\gamma)-\delta)}$,
$\gamma\in(0,\frac12)$.

\begin{rmk}
(a) The closest previous result that we could find for dynamical systems is due to Grama {\em et al.}~\cite{GramaPagePeigne14}.  Their method, 
 is based on a result of Gou\"ezel~\cite{Gouezel10} which applies to dynamical systems with spectral gaps, so it is plausible that~\cite{GramaPagePeigne14} yields the convergence rate $n^{-(\frac14-\delta)}$ in the WIP for uniformly hyperbolic maps (though no such claim is explicitly made in their paper).
The results in~\cite{GramaPagePeigne14} do not apply to nonuniformly hyperbolic systems such as~\eqref{eq:LSV}.
\\[0.75ex]
(b) The convergence rates that we obtain are certainly not optimal, but this is the typical situation even in the probability literature as soon as one moves outside of the iid setting, see Remark~\ref{rmk:optimal} and references therein.
\end{rmk}

Next, we consider applications to fast-slow systems of the form
\begin{align} \label{eq:fastslow}
x_\eps(n+1)=x_\eps(n)+\eps^2a_\eps(x_\eps(n),y(n))+\eps b(x_\eps(n))v(y(n)), \quad x_\eps(0)=\xi, 
\end{align}
where $a_\eps:\R\times \Lambda\to\R$, $b:\R\to\R$, $v:\Lambda\to\R$ satisfy 
mild conditions, $\xi\in\R$, and the fast variables $y(n)\in\Lambda$ are generated by iterating a nonuniformly hyperbolic dynamical system.
In~\cite{GM13b}, it was shown that the slow variables $x_\eps(n)$, suitably scaled in time, converge to the solution of the stochastic differential equation (SDE)
given in~\eqref{eq:XSDE}.
In this paper, we obtain the first estimates of the rate of convergence.
For uniformly hyperbolic fast dynamics, we obtain the convergence rate $\eps^{\frac13-\delta}$ (note that~$\eps$ is identified with $n^{-\frac12}$  so this corresponds to $n^{-(\frac16-\delta)}$).  We also obtain convergence rates for nonuniformly hyperbolic fast dynamics including the intermittent map~\eqref{eq:LSV} for all $\gamma\in(0,\frac12)$.  Moreover, for 
$\gamma\in[\frac{1}{12}(11-\sqrt{73}),\frac12)$ 
we obtain the same rate $\eps^{\frac12(1-2\gamma)-\delta}$ as in the WIP.

\begin{rmk}
This paper is based on results of the first author in his Ph.\ D. thesis~\cite{AntoniouPhD} which concentrates on uniformly expanding maps but via a method which, as demonstrated in this paper, generalizes to nonuniformly hyperbolic transformations.  
\end{rmk}

The remainder of the paper is organized as follows.  In Section~\ref{sec:main}, we define a class of nonuniformly expanding maps and state precisely our results on convergence rates in the WIP and for homogenization of fast-slow systems.
In Section~\ref{sec:KKM}, we summarize some recent results of~\cite{KKM18} on martingale approximations.
The main result for the WIP is proved in Section~\ref{sec:WIP} and
the main result for fast-slow systems is proved in Section~\ref{sec:fastslow}.
The extension to nonuniformly hyperbolic systems is covered in Section~\ref{sec:NUH}.
In Section~\ref{sec:ex}, we discuss examples where our results apply.

\vspace{-2ex}
\paragraph{Notation}

We use ``big O'' and $\ll$ notation interchangeably, writing $a_n=O(b_n)$ or $a_n\ll b_n$
if there is a constant $C>0$ such that
$a_n\le Cb_n$ for all $n\ge1$.

Recall that $v:\Lambda\to\R$ is a H\"older observable on a metric space $(\Lambda,d)$ with exponent $\eta\in(0,1]$, written $v\in C^\eta(\Lambda)$, if
  $\|v\|_{\eta} = |v|_\infty + |v|_\eta<\infty$ where $|v|_\infty=\sup_\Lambda|v|$ and
$\BIG |v|_\eta= \sup_{x\neq x'} \frac{|v(x)-v(x')|}{d(x,x')^\eta}$.

We denote by $C[0,1]$ the Banach space of continuous functions on $[0,1]$ equipped with the supnorm.

\section{Statement of the main results}
\label{sec:main}

Let $(\Lambda,d)$ be a bounded metric space with Borel probability measure $\rho$ and let $T:\Lambda\to\Lambda$ be a nonsingular transformation (so $\rho(T^{-1}E)=0$ if and only if $\rho(E)=0$ for all Borel sets $E\subset \Lambda$).
We assume that $\rho$ is ergodic ($\rho(E)=0$ or $1$ for all Borel sets $E\subset \Lambda$ with $T^{-1}E=E$).

Suppose that $Y\subset \Lambda$ is a subset of positive measure, and that $\alpha$ is an at most countable measurable partition of $Y$. 
Let $\tau:Y\to\Z^+$ be an integrable function, constant on partition elements, and define $F(y)=T^{\tau(y)}(y)$.  We assume that $FY\subset Y$; then $\tau$ is called a {\em return time} and $F:Y\to Y$ is the corresponding {\em induced map}.

We suppose that that there are constants $\lambda>1$, $\eta\in(0,1]$, $C>0$ such that for all $a\in\alpha$, $y,y'\in a$, $0\le\ell\le\tau(a)-1$,
\begin{itemize}
\item[(i)] $F=T^\tau$ restricts to a measure-theoretic bijection from $a$ onto $Y$.
\item[(ii)] $d(Fy,Fy')\ge \lambda d(y,y')$.
\item[(iii)] $g=d\rho|_Y/d\rho|_Y\circ F$ satisfies 
$|\log g(y)-\log g(y')|\le Cd(Fy,Fy')^\eta$.
\item[(iv)] $d(T^\ell y,T^\ell y')\le C d(Fy,Fy')$.
\end{itemize}
Such a dynamical system $T$ is called {\em nonuniformly expanding}.  
We say that 
$T$ is {\em nonuniformly expanding of order $p$} if $\tau\in L^p$.
A standard consequence of (i)--(iii) is that there is a unique 
absolutely continuous ergodic $T$-invariant probability measure~$\mu$ on $\Lambda$.

Let $C^\eta_0(\Lambda)=\{v\in C^\eta(\Lambda):\int_\Lambda v\,d\mu=0\}$.
Given $v\in C^\eta_0(\Lambda)$, we
define $v_n=\sum_{j=0}^{n-1}v\circ T^j$ for $n\ge1$.  Also, define 
\[
W_n(t)=n^{-\frac12}v_{nt},\quad\text{for $\SMALL t=\frac{j}{n}$, $0\le j\le n$},
\]
and linearly interpolate to obtain a 
process $W_n\in C[0,1]$.
The following result is well-known,
see for example~\cite{Gouezel04,KKM18,MN05}:
\begin{lemma} \label{lem:CLT}
Suppose that $T:\Lambda\to\Lambda$ is nonuniformly expanding of order $2$ and let $v\in C^\eta_0(\Lambda)$.  
Then
\begin{itemize}
\item[(a)] 
The limit $\sigma^2=\lim_{n\to\infty}n^{-1}\int_\Lambda v_n^2\,d\mu$ exists.
If in addition $\gcd\{\tau(a):a\in\alpha\}=1$ (guaranteeing that $T$ is mixing), then $\sigma^2$ is given by the absolutely summable series
\begin{align} \label{eq:GK}
\sigma^2=\int_\Lambda v^2\,d\mu+2\sum_{n=1}^\infty 
\int_\Lambda v\,v\circ T^n\,d\mu.
\end{align}
\item[(b)] Typically $\sigma^2>0$.  Indeed, there is a closed subspace $S\subset C^\eta_0(\Lambda)$ of infinite codimension such that $\sigma^2>0$ whenever $v\not\in S$.
\item[(c)] The CLT holds: $n^{-\frac12}v_n\to_d N(0,\sigma^2)$ as 
$n\to\infty$ on the probability space $(\Lambda,\mu)$.
\item[(d)]  The WIP holds: $W_n\to_w W$ in $C[0,1]$ as 
$n\to\infty$ on the probability space $(\Lambda,\mu)$, where
$W$ is Brownian motion with variance $\sigma^2$. \qed
\end{itemize}
\end{lemma}

\subsection{Rates in the WIP}

So far, we assumed only that $\tau\in L^2$.
For $\tau\in L^p$, $p>2$, Gou\"ezel~\cite{Gouezel05} obtained convergence rates (Berry-Esseen estimates) in the CLT.  
Our first main result is a convergence rate in the WIP.

The Prokhorov metric $\pi_1$ is given by
\[
\pi_1(X,Y)=\inf\{\eps>0:\P(X\in A)\le \P(Y\in A^\eps)+\eps\, \;\text{for all closed $A\in\cB$}\}.
\]
Here, $\cB$ is the Borel $\sigma$-algebra on $C[0,1]$ and
$A^\eps$ is the $\eps$-neighbourhood of $A$.
The WIP in Lemma~\ref{lem:CLT}(d) can be rewritten as $\lim_{n\to\infty}\pi_1(W_n,W)=0$.

\begin{thm} \label{thm:WIP}
Let $T:\Lambda\to\Lambda$ be nonuniformly expanding of order \mbox{$p>2$} and
suppose that $v\in C^\eta_0(\Lambda)$.  
Then there is a constant $C>0$ such that
 $\pi_1(W_n,W)\le C n^{-r(p)}$ for all $n\ge1$,
where $r(p)=\frac{p-2}{4p}$.
\end{thm}

\begin{rmk} \label{rmk:optimal}  The exact formula for the function $r:(2,\infty)\to(0,\frac14)$ can probably be improved slightly using more careful arguments.  However, it is known that the main feature, namely $\sup r=\frac14$, is essentially optimal under the methods used.

Specifically, we use a result of Kubilius~\cite{Kubilius94} which builds upon~\cite{Haeusler84,HallHeyde80}.  These results use the martingale version of the Skorokhod embedding theorem; by~\cite{Borovkov73,Sawyer72}, this method cannot yield rates better than $O(n^{-\frac14})$.   
\end{rmk}

\subsection{Rates for fast-slow systems}
\label{subsec:rates}

Next, we consider fast-slow dynamical systems of the form~\eqref{eq:fastslow}.
Here $x_\eps\in\R$ denotes the slow variables and the fast $y$-variables are generated by a nonuniformly expanding map $T:\Lambda\to\Lambda$, so
$y(n)=T^ny_0$ where $y_0$ is chosen randomly from $(\Lambda,\mu)$.

We continue to assume that $v\in C^\eta_0(\Lambda)$ and also that $a_\eps:\R\times \Lambda\to\R$ and $b:\R\to\R$, satisfy 
the following regularity conditions:
\vspace{-1ex}
\paragraph{Regularity conditions:}  
There are constants $C>0$, $\Lip\,a_0>0$, such that
\begin{itemize}

\parskip=-2pt
\item[(i)]
$\sup_\eps|a_\eps|_\infty<\infty$. \hspace{1em}
(ii)
$\,\sup_{x,y,\eps}|a_\eps(x,y)-a_0(x,y)|\le C\eps^\frac13$.
\item[(iii)] $a_0$ is Lipschitz in $x$ uniformly in $y$. That is,
$|a_0(x,y)-a_0(x',y)|\le \Lip\, a_0\,  |x-x'|$ for all
$x,x'\in\R$, $y\in\Lambda$.
\end{itemize}
\vspace{-1ex}
Moreover,
$b$ is $C^2$ and nonvanishing
with $b,b',b'',1/b\in L^\infty$.
(It is clear from the proof that $C^2$ can be reduced to $C^{\frac43}$.)

\vspace{1ex}
Let $\hat x_\eps(t)=x_\eps(t\eps^{-2})$ for $t=0,\eps^2,2\eps^2,\ldots$ and  linearly interpolate to obtain $\hat x_\eps\in C[0,1]$.
By~\cite[Theorem~1,3]{GM13b} (see also~\cite[Section~6]{KKM18}), $\hat x_\eps\to_w X$ in $C[0,1]$ for $T$ nonuniformly expanding of order $2$, where $X$ is the solution of the Stratonovich  SDE
\begin{align} \label{eq:XSDE}
dX=\Big\{
\bar a(X)-\frac12 b(X)b'(X)\int_\Lambda v^2\,d\mu
\Big\}\,dt+b(X)\circ dW, \quad X(0)=\xi.
\end{align} 
Here $W$ is Brownian motion with variance $\sigma^2$ as in Lemma~\ref{lem:CLT} and $\bar a(x)=\int_\Lambda a(x,y)\,d\mu(y)$.
Our second main result gives an estimate for the rate of convergence $\pi_1(\hat x_\eps,X)$ in the Prokhorov metric.

\begin{thm} \label{thm:NUEfs}
Let $T:\Lambda\to\Lambda$ be nonuniformly expanding of order $p>2$.
Suppose that $a_\eps$ and $b$ satisfy the above regularity conditions.
Suppose further that $v\in C^\eta_0(\Lambda)$ and that
$\sup_{x\in\R}|a_0(x,\cdot)|_\eta<\infty$.
Then there is a constant $C>0$ such that
\[
\pi_1(\hat x_\eps,X)\le \begin{cases} 
C\eps^{\frac{p-2}{2p}} & p\le p_* \\
C\eps^{\frac13\frac{2p-2}{2p-1}}(-\log\eps)^{\frac12(p-1)} & p>p_*
\end{cases},
\]
where $p_*=\frac14(11+\sqrt{73})\approx 4.89$.
\end{thm}

\begin{rmk} \label{rmk:Ito}
When $T$ is mixing, we have the alternative representation of the SDE~\eqref{eq:XSDE} in It\^{o} form
\[
dX=\Big\{
\bar a(X)+b(X)b'(X)\sum_{n=1}^\infty \int_\Lambda v\,v\circ T^n\,d\mu
\Big\}\,dt+b(X)\, dW, \quad X(0)=\xi.
\]
This is immediate from formula~\eqref{eq:GK} and the It\^{o}-Stratonovich conversion formula.
The nature of the drift coefficient and the departure from It\^{o} or Stratonovich is discussed further in~\cite{GivonKupferman04,GM13b,KupfermanPavliotisStuart04}.  
 \end{rmk}

\section{Martingale approximation}
\label{sec:KKM}

In this section, we recall some results of Gordin-type~\cite{Gordin69} on martingale approximation from~\cite{KKM18}.
The key advantage of~\cite{KKM18} over other martingale approximation methods~\cite{Gordin69,Liverani96,TyranKaminska05} is that it gives good control over sums of squares of the approximating martingale, see Proposition~\ref{prop:second} below.

Recall that a measure-preserving transformation $f:\Delta\to\Delta$ on a probability space $(\Delta,\cM,\mu_\Delta)$ is called an {\em extension} of $T:\Lambda\to\Lambda$ if there is a measure-preserving map $\pi_\Delta:\Delta\to\Lambda$ such that
$\pi_\Delta \circ f=T\circ\pi_\Delta$.  The map $\pi_\Delta$ is called a {\em semiconjugacy}.

Throughout this section, we suppose that $T:\Lambda\to\Lambda$ is a nonuniformly expanding map of order $p\ge2$ and that $\eta\in(0,1]$.
By~\cite[Propositions~2.4 and~2.5, and Corollary~2.8]{KKM18}, we can
form the following ``martingale-coboundary decomposition'':

\begin{prop} \label{prop:decomp}
There is an extension $f:\Delta\to\Delta$ of $T:\Lambda\to\Lambda$ such that
for any $v\in C^\eta_0(\Lambda)$ there exists
$m\in L^p(\Delta)$ and $\chi\in L^{p-1}(\Delta)$ with
\[
v\circ\pi_\Delta=m+\chi\circ f-\chi, \qquad \E(m|f^{-1}\cM)=0.
\]
Moreover, there is a constant $C>0$ such that for all $v\in C^\eta_0(\Lambda)$, $n\ge1$,
\[
|m|_p\le C\|v\|_\eta, \quad\text{and}\quad
\Big|\max_{j\le n}|\chi\circ f^j-\chi|\Big|_p\le C\|v\|_\eta\, n^{\frac1p}
.\quad\footnote{By~\cite{KKM18}, $\chi\circ f^j-\chi\in L^p(\Delta)$ for $j\ge 1$
even though $\chi$ is generally only $L^{p-1}$.}
 \qquad
\qed
\]
\end{prop}

\begin{rmk} The extension $\Delta$ is called the Young tower~\cite{Young99} associated to the nonuniformly expanding map $T$.  The definition and properties of $\Delta$ are not required in this paper; the current section summarizes all the results that we need.
\end{rmk}

Let
\[
\cG_{n,j}=f^{-(n-j)}\cM, \quad \xi_{n,j}=n^{-\frac12}\sigma^{-1}m\circ f^{n-j}.
\]
\begin{prop}
$\{\xi_{n,j},\cG_{n,j};\,1\le j\le n\}$ is a martingale difference array.
\end{prop}

\begin{proof}  See for example~\cite[Proposition~2.9]{KKM18}.
\end{proof}

\begin{prop} \label{prop:moment}
There is a constant $C>0$ such that
\[
\big|\max_{j\le n}|v_j|\big|_{2(p-1)}\le C\|v\|_\eta\,n^\frac12
\quad\text{for all $v\in C^\eta_0(\Lambda)$, $n\ge1$.} 
\]
\end{prop}

\begin{proof}
This result is due to~\cite{MN08,MTorok12}.  For the formulation stated here, see for example~\cite[Corollary~2.10]{KKM18}.
\end{proof}

\begin{prop} \label{prop:second}
There is a constant $C>0$ such that 
\[
\Big|n^{-1}\sum_{j=0}^{k-1}\E(m^2-\sigma^2|f^{-1}\cM)\circ f^j\Big|_{2(p-1)}
\le C\|v\|_\eta^2 \,n^{-\frac12} \quad\text{for all $v\in C^\eta_0(\Lambda)$, $n\ge1$.}
\]
\end{prop}

\begin{proof} 
This is~\cite[Corollary~3.2]{KKM18}.  
For some reason, the result is stated there only with an $L^p$ bound, but
the superior $L^{2(p-1)}$ bound here is 
immediate from the argument used in~\cite{KKM18} (specifically~\cite[Corollary~2.10, Remark~2.16 and Proposition~3.1]{KKM18}).
\end{proof}

\section{Convergence rates in the WIP}
\label{sec:WIP}

In this section, we prove Theorem~\ref{thm:WIP}.
Recall that $T:\Lambda\to\Lambda$ is a nonuniformly expanding map
of order $p>2$ and $v\in C^\eta_0(\Lambda)$.
Let $r(p)=\frac{p-2}{4p}$.

We let $C$ denote constants that may depend on~$p$ and $\|v\|_\eta$.
Also, $C'$ denotes constants that may depend on $p$ but not $v$.

Define $v\circ\pi_\Delta=m+\chi\circ f-\chi$ and
the martingale difference array
$\{\xi_{n,j},f^{-(n-j)}\cM\}$ as in Section~\ref{sec:KKM}.

For $\ell\ge1$, define $V_{n,\ell}=\sum_{j=1}^\ell \E(\xi_{n,j}^2|\cG_{n,j-1})$.
\begin{prop} \label{prop:V}
$\big|\max_{\,k\le n}|V_{n,k}-\frac{k}{n}|\big|_{2(p-1)}\le C'\|v\|_\eta^2 \,n^{-\frac12}$.
\end{prop}

\begin{proof}
We have
\begin{align*}
V_{n,k}-\frac{k}{n} &=
\sigma^{-2}n^{-1}\sum_{j=1}^k\E(m^2\circ f^{n-j}|f^{-(n-j+1)}\cM)-\frac{k}{n}
\\ & =\sigma^{-2}n^{-1}\sum_{j=0}^{k-1}\E(m^2-\sigma^2|f^{-1}\cM)\circ f^j.
\end{align*}
The result follows from Proposition~\ref{prop:second}.
\end{proof}

For each $n\ge1$, define
\begin{align} \label{eq:Xn}
X_n(t)=\sum_{j=1}^k \xi_{n,j}, \quad\text{for $t=V_{n,k}/V_{n,n}$, $0\le k\le n$},
\end{align}
 and linearly interpolate to obtain a process $X_n\in C[0,1]$.
We use the following result of~\cite{Kubilius94} for martingale difference arrays to estimate the rate of convergence of $X_n$ to the unit Brownian motion $B$.

\begin{thm}[ {Kubilius~\cite[Theorem 1]{Kubilius94}} ] \label{thm:Kub}
Let $\delta\in[0,\frac34]\cup\{1\}$.
There is a constant $C>0$ such that
$\pi_1(X_n,B)\le C\lambda|\log\lambda|$
where $\lambda=\lambda_1+\lambda_2$ and
\begin{align*}
\lambda_1 & =\inf_{0\le\eps\le1}\big\{\eps^{\frac12}+\bigl(\E\sum_{j=1}^n|\xi_{n,j}|^{2+2\delta}1_{\{|\xi_{n,j}|>\eps\}}\big)^{1/(3+2\delta)}\big\}, \\
\lambda_2 & =\inf_{0\le\eps\le1}\big\{\eps+\P(|V_{n,n}-1|>\eps^2)\big\}.
\end{align*}

\vspace{-5ex}
\qed
\end{thm}

\begin{lemma} \label{lem:Xn}
$\pi_1(X_n,B)\le C n^{-r(p)}$.
\end{lemma}

\begin{proof}
Let $\lambda=\lambda_1+\lambda_2$ be as in Theorem~\ref{thm:Kub}.
We claim that
\begin{align*}
\lambda_1 & \ll \|v\|_\eta^{r'} n^{-r_1(p)},  \qquad 
\lambda_2  \ll \|v\|_\eta^{(4p-4)/(4p-3)} n^{-(p-1)/(4p-3)},
\end{align*}
where $r_1(p)=\begin{cases} \frac{p-2}{2p+2} & 2< p\le \frac72 \\
\frac{p-2}{4p-5} & \frac72\le p<4 \\ \frac{p-2}{4p-6} & p\ge4
\end{cases}$.
Then $\lambda_2\ll\lambda_1$ and
it follows from Theorem~\ref{thm:Kub} that
$\pi_1(X_n,B)\ll n^{-r_1(p)}\log n$.
In all cases, $r_1(p)>r(p)$ so the result follows.

First, we verify the estimate for $\lambda_1$.
Choose $\delta\in[0,\frac34]\cup\{1\}$ greatest such that $2+2\delta\le p$.
In other words,
$\delta=\begin{cases}   \frac{p-2}{2} & 2\le p\le \frac72 \\ \frac34 & \frac72\le p<4 \\ 1 & p\ge 4 \end{cases}$.
By stationarity,
\[
\E\sum_{j=1}^n|\xi_{n,j}|^{2+2\delta}1_{\{|\xi_{n,j}|\ge\eps\}}
=\sigma^{-(2+2\delta)}n^{-\delta}\E(|m|^{2+2\delta}1_{\{|m|\ge \eps\sigma n^{1/2}\}}).
\]
By H\"older's inequality, and then Markov's inequality,
\begin{align*}
\sigma^{2+2\delta}\E\sum_{j=1}^n|\xi_{n,j}|^{2+2\delta}1_{\{|\xi_{n,j}|\ge\eps\}}
& \le n^{-\delta} |m|_p^{2+2\delta}\, \mu(|m|\ge \eps\sigma n^{1/2})^{(p-2-2\delta)/p}
\\ & \le n^{-\delta} |m|_p^{2+2\delta}\, \Big(\frac{|m|_p^p}{\eps^p\sigma^p n^{p/2}}\Big)^{(p-2-2\delta)/p}
\\ & =\sigma^{-(p-2-2\delta)}|m|_p^p\,\eps^{-(p-2-2\delta)}n^{-(p-2)/2}.
\end{align*}
Hence
\begin{align*}
\lambda_1 & \ll \inf_{0\le\eps\le1}\{\eps^{1/2}+\|v\|_\eta^{p/(3+2\delta)}\eps^{-(p-2-2\delta)/(3+2\delta)}n^{-(p-2)/(6+4\delta)}\}
\\ & \le 2\|v\|_\eta^{p/(2p-2\delta-1)}n^{-(p-2)/(4p-4\delta-2)}
 = \|v\|_\eta^{p/(2p-2\delta-1)}n^{-r_1(p)}.
\end{align*}

Second, we verify the estimate for $\lambda_2$.
By Proposition~\ref{prop:V} and Markov's inequality,
$\P(|V_{n,n}-1|>\eps^2)\ll \|v\|_\eta^{4(p-1)}\eps^{-4(p-1)}n^{-(p-1)}$.
Hence
\[
\lambda_2\ll \inf_{0\le\eps\le1}\{\eps+\|v\|_\eta^{4(p-1)}\eps^{-4(p-1)}n^{-(p-1)}\}
\le 2\|v\|_\eta^{(4p-4)/(4p-3)}n^{-(p-1)/(4p-3)},
\]
completing the proof.
\end{proof}

The integer $k$ in~\eqref{eq:Xn} is a random variable
$k=k_{n,t}:\Delta\to\{0,\dots,n\}$ given by 
\[
V_{n,k}\le t V_{n,n} < V_{n,k+1}.
\]

\begin{prop}  \label{prop:k}
$\big|\supI|k-[nt]|\big|_{2(p-1)} \le Cn^{\frac12}$.
\end{prop}

\begin{proof}
Set $\tV_{n,j}=nV_{n,j}-j$.
Then
$\tV_{n,k}+k\le t\tV_{n,n}+nt < \tV_{n,k+1}+k+1$,
and it follows that $k-nt$ satisfies the inequalities
\[
k-nt\le t\tV_{n,n}-\tV_{n,k}\le \tV_{n,n},
\]
and
\[
k-nt>
t\tV_{n,n}-\tV_{n,k+1}-1\ge
-\tV_{n,k+1}-1.
\]
Hence
\[
|k-[nt]|\le |k-nt|+1\le \max_{j\le n+1}|\tV_{n,j}|+2
=n\max_{j\le n+1}|V_{n,j}-{\SMALL\frac{j}{n}}|+2,
\]
and so the result follows from Proposition~\ref{prop:V}.
\end{proof}

\subsection{Passing from $X_n$ to $W_n$}

\begin{prop} \label{prop:pi}
Let $Y$, $Y'\in C[0,1]$ be random elements defined on a common probability space, and let $\eps_0,\,\eps_1>0$, $q\ge1$.
\\[.75ex]
(a) If 
$\P\big(\supI|Y-Y'|\ge\eps_0)\le\eps_1$, then 
$\pi_1(Y,Y')\le\max\{\eps_0,\eps_1\}$.
\\[.75ex]
(b)
If 
$\big|\supI|Y-Y'|\big|_q\le \eps_0$,
then 
$\pi_1(Y,Y')\le \eps_0^{q/(q+1)}$.
\end{prop}

\begin{proof}
(a) 
This is immediate from the definition of $\pi_1$.
\\[.75ex]
(b)
By Markov's inequality,
$\P(\supI|Y-Y'|\ge \eps)\le \eps^{-q}\eps_0^q$.
In particular, $\P(\supI|Y-Y'|\ge \eps_1)\le\eps_1$
if $\eps_0$ satisfies $\eps_1^{-q}\eps_0^q =\eps_1$.
In other words,
$\eps_1 =\eps_0^{q/(q+1)}$.
Hence the result follows from part~(a).
\end{proof}

\begin{prop} \label{prop:Z}
\footnote{This estimate was suggested to us by Alexey Korepanov.}
For $n\ge1$, define $Z_n=\max_{\,0\le i,\ell\le n^{\frac12}}|v_\ell|\circ T^{i[n ^{\frac12}]}$.
Then
\begin{itemize}
\item[(a)] $|\sum_{j=a}^{b-1}v\circ T^j|\le Z_n((b-a)(n^{\frac12}-1)^{-1}+3)$
for all $0\le a<b\le n$.
\item[(b)] 
 $|Z_n|_{2(p-1)}\le C'\|v\|_\eta\, n^{\frac14+\frac{1}{4(p-1)}}$ for all $n\ge1$.
\end{itemize}
\end{prop}

\begin{proof}
(a)  Choose $0\le\ell_1\le\ell_2\le \sqrt n$ greatest such that
$\ell_1[\sqrt n]\le a$ and $\ell_2[\sqrt n]\le b$.
Then 
$\ell_2-\ell_1\le 
\frac{b-a}{\sqrt n-1}+1$.
It follows that
\begin{align*}
\Big|\sum_{j=a}^{b-1}v\circ T^j\Big|
\le (\ell_2-\ell_1+2)Z_n \le \Big(\frac{b-a}{\sqrt n-1}+3\Big)Z_n,
\end{align*}
as required.
\\[.75ex]
(b) We have
\begin{align*}
\int_\Delta|Z_n|^{2(p-1)}\,d\mu_\Delta & \le 
\sum_{i\le n^{1/2}}\int_\Delta\max_{\ell\le n^{1/2}} |v_\ell|^{2(p-1)}\,d\mu_\Delta
 \ll n^{\frac12}\int_\Delta\max_{\ell\le n^{1/2}} |v_\ell|^{2(p-1)}\,\mu_\Delta.
\end{align*}
By Proposition~\ref{prop:moment},
\begin{align*}
|Z_n|_{2(p-1)} & \ll n^{\frac{1}{4(p-1)}}\Big|\max_{\ell\le n^{1/2}} |v_\ell|\Big|_{2(p-1)}
 \ll \|v\|_\eta \,n^{\frac14+\frac{1}{4(p-1)}}. 
\end{align*}
\end{proof}

Following~\cite[Lemma~4.8]{KM16}, we
define the linear functional 
\[
g:C[0,1]\to C[0,1], \qquad g(u)(t)=u(1)-u(1-t).
\]

\begin{lemma}  \label{lem:Mn}
$\pi_1(g\circ W_n\circ\pi_\Delta,\sigma X_n) \le Cn^{-r(p)}$.
\end{lemma}

\begin{proof}
Define the piecewise constant
process $V_n'(t)=n^{-\frac12}\sum_{n-[nt]}^{n-k-1}v\circ T^j$.
Then
\begin{align*}
g & \circ W_n(t)  \circ\pi_\Delta-\sigma X_n(t)
 = n^{-\frac12}\sum_{j=n-[nt]}^{n-1}v\circ T^j\circ\pi_\Delta-n^{-\frac12}\sum_{j=1}^k m\circ f^{n-j}+E_n(t)
\\ 
 &=n^{-\frac12}\sum_{j=n-[nt]}^{n-1}v\circ T^j\circ\pi_\Delta - n^{-\frac12}
\Big(\sum_{j=n-k}^{n-1}v\circ \pi_\Delta\circ f^j-\chi\circ f^n+\chi\circ f^{n-k}\Big) +E_n(t)
\\ &
 =V_n'(t)\circ \pi_\Delta+
n^{-\frac12}(\chi\circ f^{n-k}-\chi\circ f^n)+ E_n(t),
\end{align*}
where $|E_n(t)|_\infty\le n^{-\frac12}|v|_\infty$.
By Proposition~\ref{prop:decomp},
\[
\big|n^{-1/2}\supI|\chi\circ f^{n-k}-\chi\circ f^n|\big|_p
\le 2n^{-\frac12}\big|\max_{j\le n}|\chi\circ f^j-\chi|\big|_p
\ll n^{-(\frac12-\frac1p)}.
\]
By Propositions~\ref{prop:k} and~\ref{prop:Z}, and Cauchy-Schwarz,
\begin{align*}
 \big|\supI|V_n'(t)\circ\pi_\Delta|\big|_{p-1} & =
 \big|\supI|V_n'(t)|\big|_{p-1}
  \le n^{-\frac12}\big|Z_n\big( n^{-\frac12}\supI|[nt]-k|+3\big)\big|_{p-1}
\\ & \le n^{-\frac12}|Z_n|_{2(p-1)}\,\big( n^{-\frac12}\big|\supI|[nt]-k|\big|_{2(p-1)}+3\big)
\\ & \ll n^{-\frac12}|Z_n|_{2(p-1)} \ll n^{-(\frac14-\frac{1}{4(p-1)})}
 = n^{-\frac14\frac{p-2}{p-1}}.
\end{align*}
Combining these estimates, we obtain $\big|\supI|g\circ W_n\circ\pi_\Delta-\sigma X_n|\big|_{p-1}
\ll n^{-\frac14\frac{p-2}{p-1}}$.
Now apply Proposition~\ref{prop:pi}(b).
\end{proof}

\begin{pfof}{Theorem~\ref{thm:WIP}}
It is easy to see that $g(W)=_d W=\sigma B$.
Since $\pi_\Delta$ is a semiconjugacy, $W_n\circ\pi_\Delta=_d W_n$.
Hence combining Lemma~\ref{lem:Xn} and Lemma~\ref{lem:Mn},
\begin{align*}
\pi_1(g\circ W_n,g\circ W) & =\pi_1(g\circ W_n\circ\pi_\Delta,\sigma B)
\\ & \le
\pi_1(g\circ W_n\circ\pi_\Delta,\sigma X_n)+\pi_1(\sigma X_n,\sigma B)
\ll n^{-r(p)}.
\end{align*}

Now $g\circ g={\rm Id}$.  Also, 
$g:C[0,1]\to C[0,1]$ is Lipschitz with $\Lip\, g\le 2$. That is $\supI|g(u)-g(v)|\le 2\,\supI|u-v|$ for all $u,v\in C[0,1]$.
Hence it follows from the Lipschitz mapping theorem~\cite[Theorem~3.2]{Whitt74} that
$\pi_1(W_n,W)=\pi_1(g(g\circ W_n),g(g\circ W))
\le 2\pi_1(g\circ W_n,g\circ W)\ll n^{-r(p)}$.
\end{pfof}

\section{Convergence rates for homogenization}
\label{sec:fastslow}

In this section, we prove Theorem~\ref{thm:NUEfs}.
We begin by proving an abstract result, Theorem~\ref{thm:fastslow} below, before specializing to the case where $T$ is nonuniformly expanding.

Let $T:\Lambda\to\Lambda$ be a map with ergodic invariant probability measure $\mu$.
Define $y(n)=T^ny_0$ where $y_0$ is chosen randomly from $(\Lambda,\mu)$, and fix $\xi\in\R$.
We consider fast-slow systems of the form~\eqref{eq:fastslow}
where $a_\eps:\R\times \Lambda\to\R$ and $b:\R\to\R$ satisfy 
the regularity assumptions listed in Section~\ref{subsec:rates} and $v\in L^\infty(\Lambda)$.

Define
\[
\SMALL \bar a(x)=\int_\Lambda a_0(x,y)\,d\mu(y),
\qquad \tilde a(x,y)=a_0(x,y)-\bar a(x).
\]
Let $\tilde a_u(y)=\tilde a(u,y)$.  We suppose that there exist $C>0$, $q\ge1$ such that
\begin{align} \label{eq:au}
 \SMALL |\sum_{j=0}^{n-1}\tilde a_u\circ T^j|_q & \le Cn^{\frac12}
\quad\text{for all $n\ge1$, $u\in\R$.}
\\
 \SMALL |\sum_{j=0}^{n-1}v\circ T^j|_q & \le Cn^{\frac12}
\quad\text{for all $n\ge1$.}
\label{eq:vn}
\end{align}
When $b\not\equiv1$, we require in addition that
\begin{align} \label{eq:v2}
\SMALL |\sum_{j=0}^{n-1}v^2\circ T^j-n\int_\Lambda v^2\,d\mu|_q\le Cn^{\frac12}
\quad\text{for all $n\ge1$.}
\end{align}

Define 
$W_\eps(t)=\eps v_{t\eps^{-2}}$ for $t=0,\eps^2,2\eps^2,\ldots$
and linearly interpolate (and restrict to $[0,1]$) to obtain $W_\eps\in C[0,1]$.
We assume that
$W_\eps\to_w W$ in $C[0,1]$ where $W$ is a one-dimensional Brownian motion with variance~$\sigma^2$.
Let $\hat x_\eps(t)=x_\eps(t\eps^{-2})$ for $t=0,\eps^2,2\eps^2,\ldots$ and  linearly interpolate to obtain $\hat x_\eps\in C[0,1]$.
By~\cite[Theorem~1,3]{GM13b}, $\hat x_\eps\to_w X$ in $C[0,1]$ where 
$X$ is the solution of the Stratonovich  SDE~\eqref{eq:XSDE}.

\begin{thm}  \label{thm:fastslow}
Suppose that $\pi_1(W_\eps,W)=O(\eps^r)$ for some $r>0$.
Then $\pi_1(\hat x_\eps,X)\le 
C( \eps^r+\eps^{\frac13\frac{q}{q+1}} (-\log\eps)^{\frac{q}{4}})$.
\end{thm}

\begin{rmk} It is easily seen from the proof that in the special case $a_\eps(x,y)\equiv \bar a(x)$, $b\equiv1$, we obtain the same rate
$\pi_1(\hat x_\eps,X)=O(\eps^r)$ as in the WIP.
\end{rmk}

\begin{pfof}{Theorem~\ref{thm:NUEfs}}
In the definitions of $W_n$ and $W_\eps$, notice that $\eps$ is identified with $n^{-\frac12}$.  Hence it follows from Theorem~\ref{thm:WIP} that
$\pi_1(W_\eps,W)=O(\eps^{\frac{p-2}{2p}})$.

Next, we verify the moment conditions~\eqref{eq:au}--\eqref{eq:v2} with $q=2(p-1)$.  Since $v$ and $v^2-\int_\Lambda v^2\,d\mu$ lie in $C^\eta_0(\Lambda)$, 
conditions~\eqref{eq:vn} and~\eqref{eq:v2} follow from Proposition~\ref{prop:moment}.  The assumption on $a_0$ implies that $\tilde a_u\in C^\eta_0(\Lambda)$ for all $u\in\R$ and that $\sup_u \|\tilde a_u\|_\eta<\infty$, so~\eqref{eq:au} also follows from Proposition~\ref{prop:moment}.

We have now verified all of the hypotheses of Theorem~\ref{thm:fastslow} and it follows that
$\BIG\pi_1(\hat x_\eps,X)\le C \big\{\eps^{\frac{p-2}{2p}}
+\eps^{\frac13\frac{2p-2}{2p-1}}(-\log\eps)^{\frac12(p-1)}\big\}$.
\end{pfof}

In Subsection~\ref{sec:b=1}, we prove Theorem~\ref{thm:fastslow} when $b\equiv1$.  The general case is proved in Subsection~\ref{sec:gen} by reducing to the case $b\equiv1$.

\subsection{The case $b\equiv1$}
\label{sec:b=1}

In this subsection, we prove Theorem~\ref{thm:fastslow} in the special case where $b\equiv1$.
In this case, the limiting SDE takes the form
\begin{align} \label{eq:SDEb=1}
dX=\bar a(X)\,dt+dW. \qquad X(0)=\xi.
\end{align}

\begin{thm} \label{thm:b=1}
Suppose that $\pi_1(W_\eps,W)=O(\eps^r)$ for some $r>0$.
Then $\pi_1(\hat x_\eps,X)\le 
C( \eps^r+\eps^{\frac13\frac{q}{q+1}} (-\log\eps)^{\frac{q}{4}})$.
\end{thm}

We have the following preliminary calculation:
Set $M=[\eps^{-\frac43}]$.

\begin{prop} \label{prop:prelim}
Suppose that $b(x)\equiv1$.  Then
\[
\hat x_\eps(t)=\xi+ 
\int_0^t\bar a(\hat x_\eps(s))\,ds +
W_\eps(t) +D_\eps(t)+ E_\eps(t), 
\]
where 
\[
D_\eps(t)=\eps^\frac23
\!
\sum_{n=0}^{[t\eps^{-\frac23}]-1}
\!
J_\eps(n), \qquad
J_\eps(n)=\eps^{\frac43}
\!\!
\sum_{j=nM}^{(n+1)M-1}
\!\!
\tilde a(x_\eps(nM),y(j)),
\]
and
$\big|\supI|E_\eps|\big|_q \le C\eps^{\frac13}$.
\end{prop}

\begin{proof}
Introduce the step function $\tilde x_\eps(t)=x_\eps([t\eps^{-2}])$.
Then 
\[
\tilde x_\eps(t)=\xi+ 
\eps^2\sum_{j=0}^{[t\eps^{-2}]-1}a_\eps(x_\eps(j),y(j))
+\eps\sum_{j=0}^{[t\eps^{-2}]-1}v(y(j)).
\]
Using the estimates
\begin{align*}
& |\hat x_\eps(t)-\tilde x_\eps(t)|_\infty\le \eps^2|a_\eps|_\infty+\eps|v|_\infty, \qquad
\SMALL  |W_\eps(t)-\eps\sum_{j=0}^{[t\eps^{-2}]-1}v(y(j))|_\infty\le \eps|v|_\infty, \\
& \SMALL |\eps^2\sum_{j=0}^{[t\eps^{-2}]-1}\{a_\eps(x_\eps(j),y(j))-
a_0(x_\eps(j),y(j))\}|_\infty=O(\eps^\frac13),
\end{align*}
we obtain
\begin{align} \label{eq:prelim}
\hat x_\eps(t) 
& =\xi+  \eps^2\sum_{j=0}^{[t\eps^{-2}]-1}a_0(x_\eps(j),y(j))+W_\eps(t)+O(\eps^\frac13)
\\
& =\xi+  
\eps^2\sum_{j=0}^{[t\eps^{-2}]-1}\bar a(\hat x_\eps(\eps^2 j))+W_\eps(t)+
F_\eps(t)+O(\eps^\frac13),
\nonumber
\end{align}
where 
$F_\eps(t)=\eps^2\sum_{j=0}^{[t\eps^{-2}]-1}\tilde a(x_\eps(j),y(j))$.

If $t\eps^{-2}$ is an integer, then 
$\eps^2\sum_{j=0}^{[t\eps^{-2}]-1}\bar a(\hat x_\eps(\eps^2 j))=\int_0^t \bar a(\hat x_\eps(s))\,ds$, while
in general $\big|\eps^2\sum_{j=0}^{[t\eps^{-2}]-1}\bar a(\hat x_\eps(\eps^2 j))-\int_0^t \bar a(\hat x_\eps(s))\,ds\big|_\infty\le \eps^2|a_0|_\infty$.
Hence
\[
\hat x_\eps(t)=\xi+ 
\int_0^t\bar a(\hat x_\eps(s))\,ds +
W_\eps(t) +F_\eps(t)+ O(\eps^\frac13).
\]

Next,
for $nM\le j<(n+1)M$, $y\in\Lambda$,
\begin{align*}
|\tilde a(x_\eps(j),y)-\tilde a(x_\eps(nM,y)|
& \le 2\,\Lip\, a_0\,|x_\eps(j)-x_\eps(nM)|
\\ & \le 2\,\Lip\, a_0\,(\eps^2 M|a_\eps|_\infty+\eps|\SMALL\sum_{i=nM}^{j-1} v(T^iy_0)|).
\end{align*}
By~\eqref{eq:vn},
$|\sum_{i=nM}^{j-1} v\circ T^i|_q
=|\sum_{i=0}^{j-nM-1}v\circ T^i|_q\ll(j-nM)^\frac12\le M^\frac12$.
Hence 
$|\tilde a(x_\eps(j),y(j))-\tilde a(x_\eps(nM,y(j))|_q=O(\eps M^\frac12)=O(\eps^{\frac13})$ uniformly in
$nM\le j<(n+1)M$.
It follows that
\[
\eps^2\Big|\sup_{t\in[0,1]}\Big|\sum_{j=0}^{[t\eps^{-\frac23}]M-1}\tilde a(x_\eps(j),y(j))
-\sum_{n=0}^{[t\eps^{-\frac23}]-1}\sum_{j=nM}^{(n+1)M-1}
\tilde a(x_\eps(nM),y(j))\Big|\,\Big|_q=O(\eps^{\frac13}).
\]
Also,
\[
\Big|F_\eps(t)-\eps^2\sum_{j=0}^{[t\eps^{-\frac23}][\eps^{-\frac43}]-1}\tilde a(x_\eps(j),y(j))\Big|_\infty\le \eps^2(\eps^{-\frac43}+\eps^{-\frac23})|\tilde a|_\infty
\le 4\eps^{\frac23}|a_0|_\infty.
\]
Combining these last two estimates, we obtain
$\big|\supI|F_\eps-D_\eps|\big|_q=O(\eps^\frac13)$ and the result follows.
\end{proof}

\begin{lemma} \label{lem:Q}
Let $Q=(-32\sigma^2\log\eps)^{\frac12}$.  Then
$\mu\big(\supI|\hat x_\eps|\ge Q\big) \le C \eps^r$.
\end{lemma}

\begin{proof}  
By the reflection principle,
\begin{align*}
\P\big(\supI|W|\ge \SMALL\frac14 Q\big) 
& \SMALL  \le 2\P\big(\supI W\ge  \frac14 Q\big) = 4\P(W(1)\ge \frac14 Q)
\\ & \SMALL  \ll \int_{Q/4}^\infty e^{-x^2/(2\sigma^2)}\,dx\ll Q^{-1}e^{-Q^2/(32\sigma^2)}\le\eps.
\end{align*}

By assumption, $\pi_1(W_\eps,W)=O(\eps^r)$.  In particular,
for $\eps$ sufficiently small,
$\mu\big(\supI|W_\eps|\ge \frac12 Q\big)\le \P(\supI|W|\ge \frac14 Q)+O(\eps^r)$.

By~\eqref{eq:prelim},
$|\hat x_\eps(t)-W_\eps(t)|\le |\xi|+|a_0|_\infty+O(\eps^\frac13)$.
Hence for~$\eps$ sufficiently small,
\begin{align*}
\mu\big(\supI|\hat x_\eps|\ge Q\big)
& \le \mu\big(\supI|W_\eps|\ge \SMALL \frac12 Q\big)
\le  \P\big(\supI|W|\ge \frac14 Q\big)+O(\eps^r)=O(\eps^r),
\end{align*}
as required.
\end{proof}

Let $B_\eps=\{\supI|\hat x_\eps|\le Q\}$
where $Q=Q_\eps$ is as in Lemma~\ref{lem:Q}.

\begin{lemma} \label{lem:Deps}
$\big|1_{B_\eps}\supI|D_\eps|\big|_q\le C\eps^\frac13(-\log\eps)^{\frac14}$.
\end{lemma}

\begin{proof}
For $u\in\R$ fixed, define
\[
\tilde J_\eps(n,u)(y_0)
=\eps^{\frac43}
\!\!
\sum_{j=nM}^{(n+1)M-1}
\!\!
\tilde a(u,y(j))
=\eps^{\frac43}
\!\!
\sum_{j=nM}^{(n+1)M-1}
\!\!
\tilde a_u(T^jy_0),
\]
where $\tilde a_u(y)=\tilde a(u,y)$.
Note that $\tilde J_\eps(n,u)=\tilde J_\eps(0,u)\circ T^{nM}$.
By assumption~\eqref{eq:au},
\[
|\tilde J_\eps(n,u)|_q =|\tilde J_\eps(0,u)|_q \ll 
\eps^{\frac43}M^{\frac12}\le \eps^{\frac23},
\]
uniformly in $n$ and $u$.

We can choose a partition $S=S_\eps\subset[-Q,Q]$ of finite cardinality $|S|$ such that
$\dist(x,S)\le 2Q/|S|$ for all $x\in[-Q,Q]$.
For $x\in [-Q,Q]$, there exists $u_x\in S$ such that
for all $y\in\Lambda$,
\[
|\tilde a(x,y)-\tilde a(u_x,y)|\le 2\,\Lip\, a_0\, 2Q/|S|= 4\,\Lip\, a_0\, Q/|S|.
\]
It follows that
\[
1_{B_\eps}|J_\eps(n)-\tilde J_\eps(n,u_{x_\eps(nM)})|\le 4\,\Lip\, a_0\, Q/|S|,
\]
and hence
\[
1_{B_\eps}|J_\eps(n)|
 \le {\SMALL\sum_{u\in S}}|\tilde J_\eps(n,u)|+4\,\Lip\, a_0\, Q/|S|.
\]
Choosing $|S|=[\eps^{-\frac13}(-\log\eps)^\frac14]$, 
\begin{align*}
\big|1_{B_\eps}|J_\eps(n)|\big|_q
& 
 \le {\SMALL\sum_{u\in S}}|\tilde J_\eps(n,u)|_q+4\,\Lip\, a_0\, Q/|S|
 \\ & \ll |S|\eps^{\frac23}+(-\log\eps)^{\frac12}/|S|
\ll \eps^{\frac13}(-\log\eps)^{\frac14}.
\end{align*}
It follows that
\begin{align*}
\big|1_{B_\eps}\supI|D_\eps|\big|_q
\le \eps^\frac23\sum_{n=0}^{[\eps^{-\frac23}]-1}|1_{B_\eps}|J_\eps(n)|\big|_q
\ll \eps^{\frac13}(-\log\eps)^{\frac14},
\end{align*}
as required.
\end{proof}

\begin{cor} \label{cor:DEeps}
$\pi_1(W_\eps+D_\eps+E_\eps,W)\le 
C( \eps^r+\eps^{\frac13\frac{q}{q+1}} (-\log\eps)^{\frac{q}{4}})$.
\end{cor}

\begin{proof}
By Proposition~\ref{prop:prelim} and
Lemma~\ref{lem:Deps},
$\big|\supI|1_{B_\eps}D_\eps+E_\eps|\big|_q\ll 
\eps^{\frac13}(-\log\eps)^{\frac14}$.  Hence by Markov's inequality,
\[
\mu(\supI|1_{B_\eps}D_\eps+E_\eps|>\eps^{\frac13\frac{q}{q+1}})\ll 
\eps^{\frac13\frac{q}{q+1}} (-\log\eps)^{\frac{q}{4}}.
\]
Combining this with Lemma~\ref{lem:Q}, we obtain
\[
\mu(\supI|D_\eps+E_\eps|>\eps^{\frac13\frac{q}{q+1}})\ll 
\eps^r+ \eps^{\frac13\frac{q}{q+1}} 
(-\log\eps)^{\frac{q}{4}}.
\]
By Proposition~\ref{prop:pi}(a),
\[
\pi_1(W_\eps+D_\eps+E_\eps,W_\eps)\ll 
\eps^r+ \eps^{\frac13\frac{q}{q+1}}
(-\log\eps)^{\frac{q}{4}}.
\]
Now combine this estimate with the assumption $\pi_1(W_\eps,W)=O(\eps^r)$.
\end{proof}

\begin{pfof}{Theorem~\ref{thm:b=1}}
Consider the functional 
$\cG:C[0,1]\to[0,1]$ given by \mbox{$\cG(u)=v$} where $v$ is the unique solution to
$v(t)=\xi+\int_0^t \bar a(v(s))\,ds+u(t)$.
Since $\bar a$ is globally Lipschitz, it follows from existence and uniqueness of solutions for ordinary differential equations that $\cG$ is well-defined.  By Gronwall's inequality, $\cG$ is Lipschitz with $\Lip\,\cG\le e^{\Lip\, \bar a}$.

By definition, $X(t)=\xi+\int_0^t\bar a(X(s))\,ds+W(t)$, so $X=\cG(W)$.
By Proposition~\ref{prop:prelim}, $\hat x_\eps=\cG(W_\eps+D_\eps+E_\eps)$.
Hence by
the Lipschitz mapping theorem~\cite[Theorem~3.2]{Whitt74}, 
\[
\pi_1(\hat x_\eps,X)=\pi_1(\cG(W_\eps+D_\eps+E_\eps),\cG(W))\le e^{\Lip\, \bar a}
\pi_1(W_\eps+D_\eps+E_\eps,W).
\]
 Now apply Corollary~\ref{cor:DEeps}.
\end{pfof}

\subsection{The general case}
\label{sec:gen}

Let $b=1/\psi'$ and write
$z_\eps(n)=\psi(x_\eps(n))$,
$\hat z_\eps(t)=\psi(\hat x_\eps(t))$.
Define
\[
\bar A(z)=
\psi'(\psi^{-1}(z))\bar a(\psi^{-1}(z))
+\frac12 \psi''(\psi^{-1}(z))b(\psi^{-1}(z))^2\int_\Lambda v^2\,d\mu.
\]

\begin{lemma} \label{lem:Z}
$\pi_1(\hat z_\eps,Z)\le C(
\eps^r+\eps^{\frac13\frac{q}{q+1}}(-\log\eps)^{\frac{q}{4}})$
where $dZ=\bar A(Z)\,dt+dW$ and $Z(0)=\psi(\xi)$.
\end{lemma}

\begin{proof}
The assumptions on $b$ ensure that $\psi$ is $C^3$ uniformly on $\R$.  
A calculation as in~\cite{GM13b,KKM18}, using the Taylor expansion of $\psi$, yields
\begin{align*}
& z_\eps(n+1)=z_\eps(n)+\eps^2 A_\eps(z_\eps(n),y(n))+\eps v(y(n)), \quad z_\eps(0)=\psi(\xi), 
\end{align*}
where
\[
A_\eps(z,y)=\psi'(\psi^{-1}(z))a_\eps(\psi^{-1}(z),y)
+{\SMALL\frac12} \psi''(\psi^{-1}(z))b(\psi^{-1}(z))^2v(y)^2+O(\eps),
\]
uniformly in $z$, $y$.
By the inverse function theorem, $(\psi^{-1})'=b\circ\psi^{-1}\in L^\infty$ and hence $\psi^{-1}$ is uniformly Lipschitz on $\R$.
It follows easily that $A_\eps$ inherits the regularity conditions (i)--(iii) from $a$.  By~\eqref{eq:v2},~$A_0$ inherits condition~\eqref{eq:au} from $a_0$.
Hence we can apply Theorem~\ref{thm:b=1}.~
\end{proof}

\begin{pfof}{Theorem~\ref{thm:fastslow}}
As in~\cite{GM13b}, a calculation using the definition of $X$ in~\eqref{eq:XSDE}
shows that $Z=\psi(X)$ satisfies the SDE in Lemma~\ref{lem:Z}.
The functional $\chi:C[0,1]\to C[0,1]$, $u\mapsto \psi^{-1}\circ u$, is Lipschitz with $\Lip\,\chi=\Lip\,\psi^{-1}$, so
by the Lipschitz mapping theorem, 
\[
\pi_1(\hat x_\eps,X)
=\pi_1(\chi(\hat z_\eps),\chi(Z))
\ll \pi_1(\hat z_\eps,Z).
\]
Now apply Lemma~\ref{lem:Z}.
\end{pfof}

\section{Nonuniformly hyperbolic transformations}
\label{sec:NUH}

In this section, we show how the main results in Section~\ref{sec:main} extend from nonuniformly expanding maps to transformations that are
nonuniformly hyperbolic 
in the sense of Young~\cite{Young98,Young99}.  We focus on the parts necessary for this paper, referring to~\cite{Young98,Young99} for further details.
(In particular, we do not restrict to systems with physical measures even though this is the case for most of the examples.)

Let $(\Lambda,d)$ be a bounded metric space with Borel probability measure $\mu$ and let $T:\Lambda\to\Lambda$ be an ergodic measure-preserving transformation.

\begin{defn} \label{def:NUH}
Let $p\ge1$, $\eta\in(0,1]$.
The transformation $T:\Lambda\to\Lambda$ is a {\em nonuniformly hyperbolic transformation of order $p$} if there exists a nonuniformly expanding map $\bar f:\bar\Delta\to\bar\Delta$ of order $p$ (with ergodic invariant probability measure $\bar\mu$) such that
\begin{itemize}
\item[(a)] $T:\Lambda\to\Lambda$ and $\bar f:\bar\Delta\to\bar\Delta$ have a common extension $f:\Delta\to\Delta$ with semiconjugacies
$\pi_\Delta:\Delta\to\Lambda$, $\bar\pi:\Delta\to\bar\Delta$,
\item[(b)]
There exists $\eta'\in(0,1]$ such that
for any $v\in C^\eta(\Lambda)$, there exists $\bar v\in C^{\eta'}(\bar\Delta)$ and $\chi\in L^\infty(\Delta)$ such that 
\[
v\circ\pi_\Delta=\bar v\circ\bar\pi+\psi\circ f-\psi.
\]
\item[(c)]  There is a constant $C>0$ such that for all $v\in C^\eta(\Lambda)$,
\[
\|\bar v\|_{\eta'} \le C\|v\|_\eta, \quad\text{and}\quad
|\psi|_\infty \le C\|v\|_\eta.
\]
\end{itemize}
\end{defn}

\begin{rmk}  (a) By~\cite[Propositions~5.3 and~5.4]{KKM18},
nonuniformly hyperbolic transformations modelled by Young towers with exponential tails~\cite{Young98} 
are nonuniformly hyperbolic of order $p$ (in the sense of Definition~\ref{def:NUH}) for all $p$.
Also, transformations that are modelled by Young tails with polynomial tails~\cite{Young99} are
nonuniformly hyperbolic of order $p$ provided (i) the inducing time in~\cite{Young99} lies in $L^p$, and (ii) the transformation $T$ contracts exponentially along stable leaves (see~\cite[Remark~5.1]{KKM18}).
\\[.75ex]
(b)
By~\cite{MV16}, the conclusions in Lemma~\ref{lem:CLT} hold for transformations that are nonuniformly hyperbolic of order $p\ge2$ and observables $v\in C^\eta_0(\Lambda)$.  This result does not require condition~(ii) from part~(a) of this remark.
\end{rmk}

Now define $W_n,W\in C[0,1]$ as in Section~\ref{sec:main}.  

\begin{thm} \label{thm:WIP_NUH}
Let $T:\Lambda\to\Lambda$ be nonuniformly hyperbolic of order \mbox{$p>2$} and
suppose that $v\in C^\eta_0(\Lambda)$.  
Then there is a constant $C>0$ such that
 $\pi_1(W_n,W)\le C n^{-r(p)}$ for all $n\ge1$,
where $r(p)=\frac{p-2}{4p}$.
\end{thm}

\begin{proof}
Define $\bW_n$ using the observable $\bar v$ from Definition~\ref{def:NUH}(b).
By Definition~\ref{def:NUH}(b),
 $|W_n\circ\pi_\Delta-\bW_n\circ\bar\pi|_\infty\le 2n^{-1/2}|\psi|_\infty$.
Hence by Proposition~\ref{prop:pi}(a), 
$\pi_1(W_n\circ\pi_\Delta,\bW_n\circ\bar\pi)\ll n^{-(\frac12-\delta)}$ for all $\delta>0$.
In particular, $\pi_1(W_n\circ\pi_\Delta,\bW_n\circ\bar\pi)\ll n^{-r(p)}$.
Since $\pi_\Delta:\Delta\to\Lambda$ and $\bar\pi:\Delta\to\bar\Delta$ are semiconjugacies, 
$\pi_1(W_n,\bW_n)\ll n^{-r(p)}$.

Since $\bar f$ is nonuniformly expanding of order $p$ and
$\bar v\in C^{\eta'}_0(\bar\Delta)$, it follows from
Theorem~\ref{thm:WIP} that
 $\pi_1(\bW_n,W)\ll n^{-r(p)}$.
This completes the proof.
\end{proof}

Next, consider a fast-slow system~\eqref{eq:fastslow} satisfying the regularity conditions in Section~\ref{subsec:rates}.
Define $\hat x_\eps,\,X\in C[0,1]$ as in Section~\ref{subsec:rates}.

\begin{thm} \label{thm:fs_NUH}
Let $T:\Lambda\to\Lambda$ be nonuniformly hyperbolic of order $p>2$. 
Suppose further that $v\in C^\eta_0(\Lambda)$ and that
$\sup_{x\in\R}|a_0(x,\cdot)|_\eta<\infty$.  
Then there is a constant $C>0$ such that
\[
\pi_1(\hat x_\eps,X)\le \begin{cases} 
C\eps^{\frac{p-2}{2p}} & p\le p_* \\
C\eps^{\frac13\frac{2p-2}{2p-1}}(-\log\eps)^{\frac12(p-1)} & p>p_*
\end{cases},
\]
where $p_*=\frac14(11+\sqrt{73})\approx 4.89$.
\end{thm}

\begin{proof}
As in the proof of Theorem~\ref{thm:NUEfs}, it suffices to verify the hypotheses of Theorem~\ref{thm:fastslow}.
By Theorem~\ref{thm:WIP_NUH}, $\pi_1(W_\eps,W)=O(\eps^{\frac{p-2}{2p}})$.

Next, $v$ and $v^2-\int_\Lambda v^2\,d\mu$ lie in $C^\eta_0(\Lambda)$ so
conditions~\eqref{eq:vn} and~\eqref{eq:v2} follow from Proposition~\ref{prop:moment} and Definition~\ref{def:NUH}(b,c).  Also, $\tilde a_u\in C^\eta_0(\Lambda)$ uniformly in $u\in\R$, so~\eqref{eq:au} follows from Proposition~\ref{prop:moment} and Definition~\ref{def:NUH}(b,c). 
\end{proof}

\section{Examples}
\label{sec:ex}

In this section, we list some examples to which our results apply.
The simplest class of examples are uniformly expanding maps, which are nonuniformly expanding with $\tau=1$.  A specific example is the Gauss map $T:[0,1]\to[0,1]$ given by $Tx=\frac{1}{x}-[\frac{1}{x}]$.
Our results hold for all $p$, so we obtain rates
$n^{-(\frac14-\delta)}$ in the WIP and $\eps^{-(\frac13-\delta)}$ for homogenization.

Similarly for uniformly hyperbolic maps (including nontrivial basic sets for Axiom~A diffeomorphism) we can take $\tau=1$ and $p$ arbitrarily large.

More generally, the rates $n^{-(\frac14-\delta)}$ in the WIP and $\eps^{-(\frac13-\delta)}$ for homogenization hold provided $T$ is nonuniformly expanding/hyperbolic and $\tau\in L^p$ for all $p$.  In particular, this covers all systems that are modelled by Young towers with exponential tails~\cite{Young98},
including:
\begin{itemize}
\item Planar periodic dispersing billiards with finite horizon~\cite{Young98}
and infinite horizon~\cite{Chernov99}, as well as billiards with external forcing and corners~\cite{Chernov99,SimoiToth14}.
\item Unimodal maps $T:[-1,1]\to[-1,1]$ given by $Tx=1-a x^2$, $a\in[0,2]$ satisfying the Collet-Eckmann condition~\cite{ColletEckmann83}, namely there are constants $b,c>0$ such that $|(T^n)'(1)|\ge ce^{bn}$ for all $n\ge1$.
By~\cite{Jakobson81,BenedicksCarleson85}, this condition holds for a positive Lebesgue measure set of parameters $a$.
\item
H\'enon like attractors.  The H\'enon map $T:\R^2\to\R^2$ introduced in~\cite{Henon76}
 is given by
 $(x,y)=(1-ax^2+y,bx)$ where $a,b\in\R$.
 By~\cite{BenedicksCarleson91,BenedicksYoung00}, there is a positive measure set
 of parameters $a<2$, $b$ small, such that $T$ is modelled by a Young tower with exponential tails.   
\end{itemize}

Other class of examples for which these rates hold are Viana maps.
These maps, introduced in~\cite{Viana97}, comprise a $C^3$ open class of multidimensional nonuniformly expanding maps.  For definiteness, we restrict attention to maps on $M=S^1\times\R$.
Let $T_0:M\to M$ be the map
$T_0(\theta,y)=(\lambda\theta\bmod1,a_0+a\sin 2\pi\theta-y^2)$,
where 
$\lambda\in\N$ with $\lambda\ge16$, 
$a_0$ is chosen so that $0$ is a preperiodic point for the quadratic map $y\mapsto a_0-y^2$, and $a$ is sufficiently small.
It follows from~\cite{Alves00,AlvesViana02} that 
$C^3$ maps sufficiently close to $T_0$ are nonuniformly expanding of order $p$ for all $p$.  (In fact, they are modelled by Young towers with stretched exponential tails~\cite{Gouezel06a}.)

Finally, we mention that the intermittent maps~\eqref{eq:LSV}, with parameter $\gamma\in(0,\frac12)$, are nonuniformly expanding of order $p$ for any $p<\frac{1}{\gamma}$, yielding the rates 
\[
\pi_1(W_n,W)=O(n^{-(\frac14(1-2\gamma)-\delta)}),
\qquad
\pi_1(\hat x_\eps,X)\le \begin{cases} 
C\eps^{\frac12(1-2\gamma)-\delta} & \gamma \ge \gamma_* \\
C\eps^{\frac13\frac{2-2\gamma}{2-\gamma}-\delta} & \gamma \le \gamma_*
\end{cases},
\]
where $\gamma_*=\frac{1}{12}(11-\sqrt{73})$.
The classical solenoid construction of Smale and Williams~\cite{Smale67,Williams67} can be used as in~\cite{AlvesPinheiro08} to construct nonuniformly hyperbolic intermittent solenoids with these rates.

\paragraph{Acknowledgements}
This research was supported in part by a
European Advanced Grant {\em StochExtHomog} (ERC AdG 320977).

\end{document}